\documentclass[a4paper,11pt]{article}
\usepackage[sc]{mathpazo}
\usepackage[T1]{fontenc}
\usepackage{youngtab}
\usepackage{amsthm,amsmath, amsfonts,hyperref,a4wide,amssymb}
\usepackage[usenames,dvipsnames]{xcolor} 
\usepackage{tikz}
\makeatletter
\tikzset{my loop/.style =  {to path={
  \pgfextra{}
  [looseness=12,min distance=6mm]
  \tikz@to@curve@path},font=\sffamily\small
  }}  
\makeatletter 
\usetikzlibrary{matrix,arrows,backgrounds,positioning,fit,shapes}

\newtheorem{theorem}{Theorem}
\newtheorem{lemma}{Lemma}

\newtheorem{proposition}[lemma]{Proposition}
\theoremstyle{definition}
\newtheorem{ex}{Example}
\newtheorem*{df}{Definition}
\theoremstyle{remark}
\newtheorem{rem}{Remark}

\newcommand*{\R}{\mathbb{R}}
\newcommand*{\Z}{\mathbb{Z}}

\newcommand*{\N}{\mathbb{N}}
\newcommand*{\C}{\mathbb{C}}

\newcommand*{\E}{\mathcal{E}}

\newcommand*{\oM}{\overrightarrow{\mathcal{M}}}

\newcommand*{\I}{\mathcal{I}}
\newcommand*{\F}{\mathcal{F}}
\newcommand*{\G}{\mathcal{G}}

\newcommand*{\D}{\mathcal{D}}

\newcommand*{\rk}{\mathrm{rk}}

\newcommand{\la}{\langle}
\newcommand{\ra}{\rangle}

\title{Graph parameters from symplectic group invariants}
\author{Guus Regts\thanks{University of Amsterdam, Sciencepark 105-107, 1098 XH Amsterdam, the Netherlands. Email: \texttt{guusregts@gmail.com.} Supported by a NWO Veni grant.}\and Bart Sevenster \thanks{University of Amsterdam. Email: \texttt{blsevenster@gmail.com}}}

\begin{document}
\maketitle
\begin{abstract}
\noindent In this paper we introduce, and characterize, a class of graph parameters obtained from tensor invariants of the symplectic group.
These parameters are similar to partition functions of vertex models, as introduced by de la Harpe and Jones, [P. de la Harpe, V.F.R. Jones, Graph invariants related to statistical mechanical models:
examples and problems, {\sl Journal of Combinatorial Theory}, Series B {\bf 57} (1993) 207--227]. 
Yet they give a completely different class of graph invariants.
We moreover show that certain evaluations of the cycle partition polynomial, as defined by Martin [P. Martin, {\sl Enumérations eulériennes dans les multigraphes et invariants de Tutte-Grothendieck}, Diss. Institut National Polytechnique de Grenoble-INPG; Université Joseph-Fourier-Grenoble I, 1977], give examples of graph parameters that can be obtained this way.


\noindent \textbf{M.S.C [2010]}  Primary   05C45, 15A72; Secondary 05C25, 05C31
\end{abstract}

\section{Introduction}
Partition functions of statistical (spin and vertex) models as graph parameters were introduced by de la Harpe and Jones in \cite{HJ}.
Partition functions of spin models include the number of graph homomorphisms into a fixed graph, and they play an important role in the theory of graph limits, cf. \cite{L12}.
A standard example of the partition function of a vertex model is the number of matchings.
Szegedy \cite{Sz7,Sz10} showed that the partition function of any spin model can be realized as the partition function of a vertex model. 
Partition functions of vertex models occur in several different mathematical disciplines. 
For example as Lie algebra weight systems in the theory of Vassiliev knot invariants cf.~\cite{CDM12}, as tensor network contractions in quantum information theory \cite{MS08} and as Holant problems in theoretical computer science \cite{CLX11,CGW13,W15}.

In \cite{Sz7}, Szegedy calls a vertex model an edge-coloring model and we will adopt his terminology here.
Let $\N$ include $0$, and for $k\in \N$ we denote by $[k]$ the set $\{1,\ldots,k\}$. Let $V_k$ be a $k$-dimensional vector space over $\C$ and let $V_k^*$ denote its dual space, i.e. the space of linear functions $V_k\to \C$.
Following Szegedy \cite{Sz7}, for $k\in \N$ we will call $h=(h^n)$, with $h^n$ a symmetric tensor in $(V_k^*)^{\otimes n}$ for each $n\in \N$, a \emph{$k$-color edge-coloring model} (in \cite{HJ} it is called a vertex model). We will often omit the reference to $k$.
The \emph{partition function} of $h$ is the graph parameter $p_h$ defined for a graph $G=(V,E)$ by
\begin{equation}
p_h(G):=\sum_{\phi:E\to[k]}\prod_{v\in V} h^{\deg(v)}(\bigotimes_{a\in \delta(v)}e_{\phi(a)}),	\label{eq:pf}
\end{equation}
where for $v\in V$, $\delta(v)$ denotes the set of edges incident with $v$ and where $e_1,\ldots ,e_k$ is a basis for $V_k$. (Note that since $h^{\deg(v)}$ is symmetric the order is irrelevant.)

Starting with the work of Freedman, Lov\'asz and Schrijver \cite{FLS} and Szegedy \cite{Sz7} a line of research has emerged in which partition function of spin and edge-coloring models have been characterized for several types of combinatorial structures such as, graphs \cite{LS9,S9,DGLRS12}, directed graphs \cite{DGLRS12,S15b}, virtual link diagrams \cite{RSS15} and chord diagrams \cite{S15a}.
The characterizations of partition functions of edge-coloring models revealed an intimate connection between the invariant theory of the orthogonal and general linear group and these partition functions. 
However, the symplectic group never showed up.
In this paper we will introduce, and characterize, a class of graph parameters related to tensors invariants of the symplectic group that we call \emph{skew-partition functions of edge-coloring models}.

It turns out that these skew-partition functions are most naturally defined for directed graphs, but, surprisingly, we show that when restricted to skew-symmetric tensors, one can in fact define them for undirected Eulerian graphs.
These skew-partition functions are related to `negative dimensional' tensors; see \cite{P71} in which Penrose already describes a basic example.
For suitable choice of tensors, these skew-partition functions give rise to evaluations of the cycle-partition polynomial (a normalization of the Martin polynomial \cite{Martin}) at negative even integers. 
As such, these skew-partition functions play a similar role for the cycle partition polynomial as the number of homomorphisms into the complete graph for the chromatic polynomial.

Besides their connection to the symplectic group, the introduction of skew-partition functions is also motivated by a paper of Schrijver \cite{S15}.
In \cite{S15} Schrijver characterized partition functions of edge-coloring models in terms of rank growth of edge-connection matrices.
We need some definitions to state this result.
For $k\in \N$, a \emph{$k$-fragment} is a graph which has $k$ vertices of degree one labeled $1,2,\ldots,k$. 
We will refer to an edge incident with a labeled vertex as an \emph{open end}.
Let $\F_k$ denote the collection of all $k$-fragments, where we allow multiple edges and loops. 
Then $\F_0$ can be considered as the collection of all graphs $\G$.
Moreover, any component of the underlying graph of a fragment in $\F_k$ may consist of a single edge with an empty vertex set; we call this graph the \emph{circle} and denote it by $\bigcirc$.
Define a \emph{gluing operation} $*:\F_k\times \F_k\to \G$, where, for two $k$-fragments $F_1,F_2$, $F_1*F_2$ is the graph obtained from $F_1\cup F_2$ by removing the labeled vertices and gluing the edges incident to equally labeled vertices together.
Note that by gluing two open ends of which the endpoints are labeled one creates a circle.
For a graph parameter $f:\G\to \C$, define the $k$-th \emph{connection matrix} $M_{f,k}$ by
\[
M_{f,k}(F_1,F_2):=f(F_1*F_2)
\]
for $F_1,F_2\in \F_k$.
Schrijver \cite{S15} proved
\begin{equation}\label{eq:thm lex}
\begin{array}{l}
f:\G\to \C \text{ is the partition function of an edge-coloring model if and only if } 
\\
f(\emptyset)=1, f(\bigcirc)\in \R \text{ and }\text{rk}(M_{f,k})\leq f(\bigcirc)^k\text{ for all }k\in \N.
\end{array}
\end{equation}
Since the rank of a matrix can never be negative, the conditions in \eqref{eq:thm lex} imply that $f(\bigcirc)\geq 0$.
It is however natural to ask what happens when we keep the rank conditions only for even $k$ and add the condition that $f(\bigcirc)<0$. Squaring a parameter satisfying these conditions, we automatically obtain the partition function of an edge-coloring model, as follows from \eqref{eq:thm lex}.
There are graph parameters that satisfy these conditions. 
Consider for example 
\begin{equation}
f(G)=\left\{\begin{array}{l} (-2)^{c(G)} \text{ if } G \text{ is $2$-regular,}
							\\ 0 \text{ otherwise,} 
							\end{array}\right.
\end{equation}
where $c(G)$ denotes the number of components of the graph $G$.
Setting $f(\bigcirc)=-2$, one has $\rk(M_{f,2k})\leq 4^k$. (This follows for example from our main theorem, cf. Theorem \ref{thm:main}.)
We note that no matter what value we choose for $\bigcirc$, this $f$ is never the (ordinary) partition function of an edge-coloring model, as follows from the results in \cite{DGLRS12}, see also \cite[Proposition 5.6]{R13}.
It turns out that skew-partition functions are graph parameters satisfying the modified conditions in \eqref{eq:thm lex}, and moreover, these are the only graph parameters. 
This is the content of our main theorem.
\\

\noindent{\bf Organization} 
In the next section we introduce skew-partition functions, derive some basic properties and state our main theorem, Theorem \ref{thm:main}.
In Section \ref{sec:martin} we show that evaluations at negative even integers of the cycle partition polynomial can be realized as the skew-partition function of suitable edge-coloring models. 
Section \ref{sec:framework} deals with some framework, which we need to prove Theorem \ref{thm:main}. 
This framework is similar to the framework developed in \cite{DGLRS12}. 
In particular, the connection with the invariant theory of the symplectic group will become clear there.
In Section \ref{sec:proof} we prove Theorem \ref{thm:main} and in Section \ref{sec:remarks} we conclude with some further remarks and questions.

\section{Definitions and main result}
In this section we introduce skew-partition functions and state our main result. 
In what follows, all graphs are allowed to have loops and multiple edges. 
We call a graph Eulerian if each vertex has even degree. 
In particular, the circle $\bigcirc$ is an Eulerian graph.

\subsection{Definitions and main result}
Let $G = (V,E)$ be an Eulerian graph with an Eulerian orientation $\omega$ of $E$. 
 A \emph{local ordering} $\kappa = (\kappa^-,\kappa^+)$ \emph{compatible} with  $\omega$ is at each vertex $v$ a pair of bijections, $\kappa^-_v$, from $\delta^-(v)$ to $\{1,3,...,d(v)-1\}$, the incoming arcs at $v$, and $\kappa^+_v$, from $\delta^+(v)$ to $\{2,4,...,d(v)\}$, the outgoing arcs at $v$. 
A local ordering $\kappa$ that is compatible with $\omega$ decomposes $E$ uniquely into circuits of the form $(v_1,a_1,...,a_i,v_i,a_{i+1},...,v_1)$, where $\kappa^-_{v_i}(a_i) + 1 = \kappa^+_{v_i}(a_{i+1})$ for each $i$. (Recall that a circuit is a closed walk where vertices may occur multiple times, but edges may not). We will refer to a circuit in this decomposition as a \emph{$\kappa$-circuit}.
Let $c(G,\kappa)$ be the number of $\kappa$-circuits in this decomposition.

Let $\mathcal{G}^{\#}$ be the set of triples $(G,\omega,\kappa)$, where $G$ is an Eulerian graph, $\omega$ is an Eulerian orientation of $G$ and $\kappa$ is a local ordering compatible with $\omega$.

Let $V_{2 \ell}:=\C^{2\ell}$ with standard basis $e_1,\ldots,e_{2\ell}$. Let $f_i\in V_{2 \ell}$ for $i=1,\ldots,2\ell$ be defined by,
\begin{equation}\label{eq:dual}
f_i:=\left\{\begin{array}{rl} -e_{i+\ell} &\text{ if } i\leq \ell,
\\ e_{i-\ell}& \text{ if }i>\ell.
\end{array}\right.
\end{equation}
If $W$ is a vector space, then $\Lambda^n(W)$ denotes the space of skew-symmetric $n$-tensors on $W$; note that this space is zero if $n > \text{dim(W)}$. 
Let $\bigwedge(V_{2 \ell}^*):=\bigoplus_{n=0}^{2\ell}\bigwedge^n(V_{2 \ell}^*)$.
For $h \in \bigwedge(V_{2 \ell}^*)$ and $(G,\omega,\kappa)\in \G^\#$ we now define
\begin{equation}\label{eq:def functie}
s_h((G,\omega,\kappa)) = (-1)^{c(G,\kappa)} \sum_{\phi: E(G) \rightarrow [2 \ell]} \prod_{v \in V(G)} h(\bigotimes_{i=1,3,\ldots,d(v)-1} e_{\phi(\kappa^{-1}_v(i))} \otimes f_{\phi(\kappa^{-1}_v(i+1))}).
\end{equation}

\begin{proposition}\label{prop:inv keuze}
Let $G$ be an Eulerian graph.
Then the value of $s_h((G,\omega,\kappa))$ for $(G,\omega,\kappa)\in \G^\#$ is independent of the choice of Eulerian orientation $\omega$ and compatible local ordering $\kappa$. 
\end{proposition}
We postpone the proof of this proposition to the end of this section. 
It allows us to define the following graph parameter.

\begin{df}\label{df:skew}
Let $h \in \bigwedge(V_{2 \ell}^*)$. Then we define the \emph{skew-partition function\footnote{The name skew-partition function is chosen so as to distinguish it from the ordinary partition function; it is moreover motivated by the fact that $s_h(G)$ can be seen as the contraction of a tensor network with respect to a skew-symmetric form.}} of $h$ to be the graph invariant defined for a graph $G=(V,E)$ by 
\begin{equation}
s_h(G) :=\left \{ \begin{array}{ll}
 s_h((G,\omega,\kappa))& \text{ for any $(G,\omega,\kappa) \in \mathcal{G}^{\#}$, if $G$ is Eulerian},
 \\
 0 &\text{ otherwise.}\end{array}\right.
\end{equation}
Notice that $s_h(\bigcirc) = -2 \ell$.
\end{df}
\begin{rem}
Setting $s_h(G) = 0$ for non-Eulerian graphs $G$ might seem unnatural at first, but the equivalence between $(ii)$ and $(iii)$ in Theorem \ref{thm:main} below shows that this is in fact the most natural definition for these functions. 
 \end{rem}
 Let us give an example.
\begin{ex}
Let $\ell=1$ and let $h\in \bigwedge(V_2^*)$ be defined by $h(e_1\otimes e_2) = -1$ and $h(e_2\otimes e_1) = 1$.
Then \[s_h(G)=\left \{\begin{array}{l}(-2)^{c(G)} \text{ if } G \text{ is } 2\text{-regular,}
\\
0\text{ otherwise.}\end{array}\right.\]
This example can be generalised quite a bit. 
Readers familiar with the cycle partition polynomial will recognise that $s_h(G)$ is equal to the cycle partition polynomial evaluated at $-2$.
In Section \ref{sec:martin} we shall show that we can realize evaluations of the cycle partition polynomial at any negative even integer as the partition function of a suitable skew-symmetric tensor.
\end{ex}
One of our main results in this paper is a characterization of graph parameters that are skew-partition functions of edge-coloring models, which we will state below after we introduce some terminology.
A graph parameter $f:\G\to \C$ is called \emph{multiplicative} if $f(F_1\cup F_2)=f(F_1)f(F_2)$ for all $F_1,F_2\in \G$ and $f(\emptyset)=1$.
For a graph $G=(V,E)$ and $U\subseteq E$ of size $\ell+1$, give every element of $U$ an orientation and label the vertices of $G$ incident with edges in $U$ as follows. 
Vertices incident with outgoing arcs are labeled with $[\ell +1]$, and vertices incident with incoming arcs are labeled with $[2\ell+2]\setminus [\ell+1]$ (some vertices may receive multiple labels this way). 
Then for $\pi\in S_{2\ell+2}$, let $G_{U,\pi}$ be the graph obtained from $G$ by removing $U$ from $E$ and adding the edges $\{\pi(i),\pi(\ell+1+i)\}$ for $i\in [\ell+1]$ to $E\setminus U$. 
(We note that the graph $G_{U,\pi}$ depends on the choice of orientation and on the labeling even though this not indicated.)
We can now state our main theorem, characterizing skew-partition functions, which we prove in Section \ref{sec:proof}.
\begin{theorem}\label{thm:main}
Let $f:\G\to \C$ be a graph parameter. Then the following are equivalent:
\begin{itemize}
\item[(i)] there exists a skew-symmetric tensor $h\in \bigwedge(V^*_{2 \ell})$ for some $\ell \in \N$ such that $f(G)=s_h(G)$,
\item[(ii)] $f$ is multiplicative, $f(\emptyset)=1$, $f(G)=0$ if $G$ is not Eulerian, and for some $\ell\in \N$, $f(\bigcirc) = -2 \ell$ and $\sum_{\pi\in S_{2\ell+2}}f(G_{U,\pi})=0$ for each graph $G=(V,E)\in \E$ and $U\subseteq E$ of size $\ell+1$,
\item[(iii)] $f(\emptyset)=1$, $f(\bigcirc)<0$ and for each $k\in \N$, the rank of $M_{f,2k}$ is bounded by $(f(\bigcirc))^{2k}$.
\end{itemize}
Moreover, the equivalence between (i) and (ii) holds with the same $\ell$.
\end{theorem}

\subsection{Proof of Proposition 1}
Before proving the proposition, we introduce a quantity that will turn out to be useful later on.
For $(G,\omega,\kappa)\in \G^\#$, $h\in \bigwedge (V_{2 \ell}^*)$ and a map $\phi:E(G)\to [\ell]$ consider
\begin{equation}\label{eq:function phi}
s_{h,\phi}(G,\omega,\kappa):=(-1)^{c(G,\kappa)}\sum_{\psi:E(G)\to \{0,\ell\}}\prod_{v\in V(G)}h(\bigotimes_{i=1,3,\ldots,d(v)-1} e_{(\phi+\psi)(\kappa^{-1}_v(i))} \otimes f_{(\phi+\psi)(\kappa^{-1}_v(i+1))}),
\end{equation}
where $(\phi+\psi):E(G)\to[2\ell]$ is defined as $e\mapsto \phi(e)+\psi(e)$ for $e\in E(G)$. Notice that $s_h((G,\omega,\kappa))$ is equal to the sum of the $s_{h,\phi}((G,\omega,\kappa))$ over all maps $\phi : E(G) \rightarrow [\ell]$.

\begin{lemma}\label{lem:inv cycle}
Let $(G,\omega,\kappa) \in \mathcal{G}^{\#}$ and let $\phi:E(G)\to [\ell]$.
Let $C$ be a $\kappa$-circuit and let $\omega'$ be the Eulerian orientation obtained from $\omega$ by inverting the orientation of the edges in $C$. Let $\kappa'$ be obtained from $\kappa$ by flipping the values of two consecutive edges of $C$ incident with a vertex $v$ for all vertices $v$ of $C$.
Then $s_{h,\phi}(G,\omega,\kappa) = s_{h,\phi}(G,\omega',\kappa')$.
\end{lemma}
\begin{proof}
Write $G=(V,E)$. We may assume that no component of $G$ is a circle.
Define for $\psi:E\to \{0,\ell\}$, $o(\psi):=|\psi^{-1}(0)|$, the number of edges that get assigned $0$ by $\psi$.
We can now, by definition of the $f_i$, rewrite \eqref{eq:function phi} as follows:
\begin{equation}\label{eq:rewrite}
(-1)^{c(G,\kappa)}\sum_{\psi:E\to \{0,\ell\}}(-1)^{o(\psi)}\prod_{v\in V(G)}h(\bigotimes_{i=1,3,\ldots,d(v)-1} e_{(\phi+\psi)(\kappa^{-1}_v(i))} \otimes e_{(\phi+\psi)(\kappa^{-1}_v(i+1))+\ell}),
\end{equation}
where the addition is carried out modulo $2\ell$. Let for $\psi : E \rightarrow \{0,\ell\}$  $\psi' : E \rightarrow \{0,\ell\}$ be defined by
\[
\psi'(e) :=\left \{\begin{array}{ll} \psi(e) &\text{ if } e \not \in C, \\ \psi(e) + \ell\mod 2\ell& \text{ if } e \in C. \end{array}\right.
\label{eq:psi kleur} 
\]
Write $H=(G,\omega,\kappa)$ and $H'=(G,\omega',\kappa')$.
We now compare the contribution of $\psi$ in $s_{h,\phi}(H)$ and $\psi'$ in $s_{h,\phi}(H')$.
Clearly, $o(\psi)$ and $o(\psi')$ differ by $|C|$.
At any vertex $v$ of $C$ with incoming arc $a$ and outgoing arc $a'$, we see $e_{(\phi+\psi)(a)}\otimes e_{(\phi+\psi)(a')+\ell}$ in $s_{h,\phi}(H)$, while in $s_{h,\phi}(H')$ we see $e_{(\phi+\psi)(a')+\ell}\otimes e_{(\phi+\psi)(a)+2\ell}$.
Since $h$ is skew-symmetric these contributions differ by a minus sign.
So by \eqref{eq:rewrite} we conclude that $\psi$ and $\psi'$ have the same contribution. This proves the lemma.
\end{proof}
We now conclude this section by proving the following lemma, which clearly implies Proposition \ref{prop:inv keuze}.

\begin{lemma}\label{lem:useful}
Let $(G,\omega,\kappa)\in \G^\#$ and let $h\in \bigwedge(V_{2 \ell}^*)$.
Then for any map $\phi:E(G)\to [\ell]$, the value $s_{h,\phi}(G,\omega,\kappa)$ is independent of the choice of $\omega$ and $\kappa$.
\end{lemma}
\begin{proof}
Suppose we have $(G,\omega,\kappa)$ and $(G,\omega,\kappa')$ in $\mathcal{G}^{\#}$. If we apply a transposition to $\kappa^+_v$ or $\kappa^-_v$ at some $v$, then the parity of the number of circuits changes. As $h$ is skew-symmetric, the evaluation of the tensor at $v$, cf. (\ref{eq:rewrite}), also changes sign, so these cancel out. So we can apply permutations at each $v$ to $\kappa$ to go from $\kappa$ to $\kappa'$ without changing the value of $s_{h,\phi}$. 

Now consider $(G,\omega,\kappa)$ and $(G,\omega',\kappa')$. The symmetric difference of $\omega$ and $\omega'$, i.e., the set of edges where they do not give the same orientation is again a Eulerian graph with Eulerian orientation given by the restriction of $\omega$. 
Let $C$ be a $\kappa$-circuit in the symmetric difference of $\omega$ and $\omega'$. 
Let $\omega''$ be obtained from $\omega$ by inverting the orientation of the edges of $C$ and let $\kappa''$ be obtained from $\kappa$ by flipping the values of two consecutive edges of $C$ incident with a vertex $v$ of $C$ for each vertex $v$ of $C$.
By Lemma \ref{lem:inv cycle} this does not change the value of $s_{h,\phi}$. 
Repeating this until there are no circuits left in the symmetric difference and then applying the first part of the proof finishes the proof.
\end{proof}

\section{Evaluations of the cycle partition polynomial}\label{sec:martin}
The cycle partition polynomial, first introduced, in a slightly different form, by Martin in his thesis ~\cite{Martin}, is related to Eulerian walks in graphs and to the Tutte polynomial of planar graphs. 
Several identities for the cycle partition polynomial were established by Bollob\'as ~\cite{Bol} and Ellis-Monaghan ~\cite{EM}.

For a graph $G = (V,E)$, let $\mathcal{C}(G)$ be the collections of all partitions of $E$ into circuits. 
For $\cal P \in \mathcal{C}(G)$, let $|\cal P|$ be the number of circuits in the partition. 
The cycle partition polynomial $J(G,x)$ is defined as
\[
J(G,x) := \sum_{\cal P \in \mathcal{C}(G)} x^{|\cal P|}. 
\]
So if $G$ is not an Eulerian graph, then $J(G,x) = 0$. We clearly have that $J(G\cup H,x) = J(G,x) J(H,x)$ and $J(\bigcirc,x) = x$. 
It is shown in~\cite{Bol,EM} that for $k \in \N$, we can express $J(G,k)$ as
\begin{equation}
J(G,k) = \sum_{A} \prod_{v \in V} \prod_{i=1}^k (\text{deg}_{A_i}(v) - 1)!!,\label{eq:cyc part pos}
\end{equation}
where $A$ ranges over ordered partitions of $E$ into $k$ subsets $A_1,\ldots,A_k$ such that $(V,A_i)$ is Eulerian for all $i$, and where, for a positive odd integer $m$, $m!!:=m(m-2)\cdots 1$.
To express \eqref{eq:cyc part pos} as the partition function of a $k$-color edge-coloring model, let $h^n\in (V_k^*)^{\otimes n}$ be defined by taking, for $\phi:[n]\to [k]$, the value $\prod_{i=1}^k (|\phi^{-1}(i)|-1)!!$ at $\bigotimes_{i\in [n]}e_{\phi(i)}$ (for convenience we set $(-1)!!=1$ and $m!!=0$ if $m$ is even).
Then $p_h(G)=J(G,k)$ for each graph $G$. 

We will similarly show that evaluations of $J(G,x)$ at negative even integers can be realized as skew-partition functions of skew-symmetric tensors. 
It follows from work of Bollob\'as \cite{Bol} (see also \cite{EM}) that the evaluation of the cycle partition polynomial $J(G,x)$ of a graph at negative even integers can be expressed as
\begin{equation}
J(G,-2 \ell)=\sum_{H_1,\ldots,H_\ell}(-2)^{\sum_{i=1}^\ell c(H_i)} \ ,\label{eq:cyc part neg}
\end{equation}
where this sum runs over all ordered partitions of the edge set of $G$ into $2$-regular subgraphs $H_1,\ldots,H_\ell$ and $c(H_i)$ denotes the number of components of the graph induced by $H_i$. 

\begin{proposition}
For every $\ell \in \N$, there is an $h \in \bigwedge (V_{2 \ell}^*)$ such that $s_h(G) = J(G,-2 \ell)$ for each graph $G$. 
\end{proposition}

\begin{proof}
Let $h \in \Lambda(V_{2 \ell}^*)$ be defined as follows. For $T \subset [\ell]$ of size $k$, and for $\pi$ in the group of permutations of $T \cup (T+\ell)$, we let
\[
h(\bigotimes_{t \in T} e_{\pi(t)} \otimes e_{\pi(t+\ell)}) = (-1)^k \text{sgn}(\pi),
\]
and let $h$ be zero on tensors that are not of this form. 
We will show that $J(G,-2\ell) = s_h(G)$. 

Take any $(G,\omega,\kappa)\in \G^\#$.
Let $\phi:E(G)\to [\ell]$ and, let for $i\in [l]$, $H_i:=\phi^{-1}(i)$. Abusing notation, the graph induced by $H_i$ will be denoted by $H_i$.
We will show that the contribution of $s_{h,\phi}(G,\omega,\kappa)$ to \eqref{eq:function phi} is equal to
\begin{equation}\label{eq:contribution}
\left\{\begin{array}{cl}0 &\text{ if some $H_i$ is not $2$-regular,}
\\ (-2)^{\sum_{i=1}^\ell c(H_i)} & \text{ if each $H_i$ is $2$-regular,}
\end{array}\right.
\end{equation}
where we use the convention that the empty graph is $2$-regular.
By \eqref{eq:cyc part neg}, this implies the proposition.

To prove \eqref{eq:contribution}, suppose first that some $H_i$ is not $2$-regular. Then $H_i$ has a vertex $v$ of odd degree.
Fix any $\psi:E(G)\to \{0,\ell\}$ and consider $\phi'=\phi+\psi$.
If $\deg(v)=1$, we cannot see both color $i$ and $i+\ell$ at $v$, implying by the definition of $h$ that the contribution of $\phi'$ to the skew-partition function is zero.
Similarly, if $\deg(v)>2$, we see color $i$ or $i+\ell$ at least twice at $v$, implying by skew-symmetry of $h$ that the contribution of $\phi'$ is zero.

Suppose now that each $H_i$ is $2$-regular.
Let us choose another Eulerian orientation $\omega'$ by choosing a cyclic orientation of all cycles of each of the  $H_i$. Take a compatible local ordering $\kappa'$, such that for two consecutive edges $a$ and $a'$ in a cycle, we have $\kappa(a')-\kappa(a) = 1$.
Then $c(G,\kappa')=\sum_{i=1}^\ell c(H_i)$.
By Lemma \ref{lem:useful} we know that $s_{h,\phi}(G,\omega,\kappa)=s_{h,\phi}(G,\omega',\kappa')$.
For $\psi:E(G)\to \{0,\ell\}$, the contribution of $\psi$ to the sum \eqref{eq:rewrite} is zero if $\psi$ is not constant on each cycle, by skew-symmetry of $h$, as by construction $\phi$ is constant on each cycle.
So let us assume that $\psi$ is constant on each cycle. 
Consider first the case where $\psi(e)=0$ for all $e\in E(G)$. Then $o(\psi)=|E(G)|$ and so the parity of $o(\psi)$ is equal the parity of the number of vertices with degree $2\mod 4$.
By definition of $h$, the $(-1)^k$ factor cancels this and so the contribution of $\psi$ is equal to $(-1)^{\sum_{i=1}^\ell c(H_i)}$.

We next show that if $\psi'$ is obtained from $\psi$ by changing the value on a cycle $C$ of some $H_i$, then they have the same contribution.
This follows from the fact that the parities of $o(\psi)$ and $o(\psi')$ differ by the parity of $|C|$ and the fact that at each vertex of $C$ we have interchanged $e_i\otimes e_{i+\ell}$ with $e_{i+\ell}\otimes e_{i}$, resulting in a factor of $(-1)^{|C|}$ by the skew-symmetry of $h$.
By \eqref{eq:rewrite}, this proves \eqref{eq:contribution} and finishes the proof of the proposition.
\end{proof}
We remark that the cycle partition polynomial evaluated at $x\notin \Z$ cannot be realized as the partition function, nor as the skew-partition function, of any edge-coloring model.
This follows from Theorem \ref{thm:main} and \eqref{eq:thm lex} and the fact that by \cite[Proposition 2]{S15} the rank of the submatrix of $M_{J(\cdot,x),2n}$ indexed by fragments in which each component consists of an unlabeled vertex with two open ends incident with it, has rank at least $n!$.
The case where $x$ is a negative odd integer is discussed in Section \ref{sec:remarks}.

\section{Framework}\label{sec:framework}
In this section we will look at locally ordered directed graphs, i.e., directed graphs in which each vertex is equipped with a total order of the arcs incident with it.
We will define skew-partition functions for these locally ordered directed graphs and characterize a class of these functions using the Nulstellensatz and the invariant theory of the symplectic group. This approach is similar to the characterization of partition functions of edge-coloring models in \cite{DGLRS12}.
This characterization will then be used to prove Theorem \ref{thm:main} in Section \ref{sec:proof}.

\subsection{Locally ordered directed graphs and skew-partition functions}
A \emph{locally ordered directed graph} is a triple $(V,A,\kappa)$, where $(V,A)$ is a directed graph (we allow multiple arcs and loops), and where $\kappa=(\kappa_v)_{v\in V}$, with $\kappa_v:[\deg(v)]\to \delta(v)$ is a bijection for $v\in V$ (here $\deg(v)$ denotes the total degree of $v$).
This gives $\delta(v)$, the arcs incident with $v$ a total ordering.
We note that in case $a$ is a directed loop at $v$, $a$ occurs twice in $\delta(v)$.
Let $\D$ denote the collection of locally ordered directed graphs. 
We often just write graph instead of locally ordered directed graph for elements of $\D$.
We call a map $f:\D\to \C$ an \emph{invariant of locally ordered directed graphs} if $f$ is constant on isomorphism classes of locally ordered directed graphs.

For $G=(V,A,\kappa)\in \D$ and $\phi:A\to [2\ell]$ (for some $\ell\in \N$), we set $o(\phi):=|\phi^{-1}([\ell])|$, the number of images of $\phi$ that land in $[\ell]$.
For $v\in V$, we let $\phi_v$ be the map $\phi_v:[\deg(v)]\to [2\ell]$ defined as follows: for $i\in [\deg(v)]$, 
\begin{equation}\label{eq:def phi_v}
\begin{array}{ll}
\phi_v(i)=\phi(\kappa_v(i))& \text{ if } \kappa_v(i)\in \delta^-(v),
\\
\phi_v(i)=\phi(\kappa_v(i))+\ell\mod 2\ell& \text{ if } \kappa_v(i)\in \delta^+(v),
\end{array}
\end{equation}
where $\delta^+(v)$ is the set of outgoing arcs and $\delta^-(v)$ is the set of incoming arcs at $v$. Let $V_{2 \ell}:=\C^{2\ell}$ for some $\ell\in \N$ and let $h=(h^n)_{n\in \N}$ with $h^n\in (V^*_{2 \ell})^{\otimes n}$; we call $h$ a \emph{$2\ell$-color edge-coloring model}. (We often omit the reference to $\ell$.)
The \emph{skew-partition function} of $h$ is the invariant of locally ordered directed graphs $s_h:\D\to \C$ defined for $G=(V,A,\kappa)\in \D$ by
\begin{equation}
s_h(G)= \sum_{\phi:A\to [2\ell]}(-1)^{o(\phi)}\prod_{v\in V}h(\phi_v),\label{eq:part skew}
\end{equation}
where $h(\phi_v)$ is equal to the value of $h^{\deg(v)}$ at $e_{\phi_v}:=e_{\phi_v(1)}\otimes \cdots \otimes e_{\phi_v(\deg(v))}$. (Here, $e_1,\ldots, e_{2\ell}$ is the standard basis of $\C^{2\ell}$.)

The name skew-partition function is motivated by the fact that if $G'$ is obtained from $G$ by flipping the direction of an arc of $G$ then $s_h(G')=-s_h(G)$, as we will see later, but which is also not difficult to show directly.

We can view a triple $(G,\omega,\kappa)\in \G^\#$, of which no component of $G$ is a circle, as a locally ordered directed graph $G'$.
In terms of the skew-partition we have for $h\in \bigwedge(V^*_{2 \ell})$ that $s_h((G,\omega,\kappa))=(-1)^{c(G,\kappa)}s_h(G')$. 
\begin{rem}
While we have defined skew-partition functions for arbitrary tensors, in the remainder of this section we will just focus on skew-symmetric tensors. The reason being that we want to give a characterization of skew-partition functions of skew-symmetric tensors, which we can use to proof Theorem \ref{thm:main}.
We remark that is possible to give a more general treatment in which one can considers edge-coloring models $(h^n)_{n\in I}$ for some subset $I\subseteq\N$, where each $h^n$ is a tensor that is invariant under some action of some subgroup of $S_n$, but for the sake of concreteness we have chosen to just stick to skew-symmetric tensors.
\end{rem}

\subsection{Invariants of the symplectic group}
Let $\langle\cdot,\cdot \rangle$ be the nondegenerate skew-symmetric bilinear form on $V_{2 \ell}$ given by $\la x,y\ra=x^TJy$ for $x,y\in V_{2\ell}$,  where $J$ is the $2\ell\times 2\ell$ matrix 
 $\left (\begin{array}{rr} 0&I\\-I&0\end{array}\right)$, with $I$ the $\ell\times \ell$ identity matrix.
Note that $\la\cdot,\cdot\ra$ naturally induces a nondegenerate symmetric bilinear form on $(V_{2 \ell})^{\otimes 2m}$ for any $m$, which again will be denoted by $\la\cdot,\cdot\ra$.

Recall that $e_1,\ldots,e_{2\ell}$ denotes the standard basis for $V_{2 \ell}$. 
Let $f_1,\ldots,f_{2\ell}\in V_{2\ell}$ be the associated dual basis with respect to the skew-symmetric form, i.e., $\la f_i,e_j\ra =\delta_{i,j}$ for all $i,j=1,\ldots, 2\ell$.
Then $f_i$ is defined by \eqref{eq:dual}.

Let $x_1,\ldots,x_{2\ell}\in V^*_{2\ell}$ be defined by $x_i(e_j)=\delta_{i,j}$.
A basis for $\bigwedge^n(V^*_{2 \ell})$ is then given by $\{x_S\mid S\subseteq [2\ell], |S|=n\}$, where for $S=\{s_1,\ldots,s_n\}$, with $s_1<s_2<\cdots<s_n$,
\[
x_S:=\sum_{\pi \in S_n}\text{sgn}(\pi) x_{s_{\pi(1)}}\otimes \cdots \otimes x_{s_{\pi(n)}}.
\]
Let $R$ be the ring of functions generated by $(\bigwedge(V_{2 \ell}^*))^*$. 
Let, for $S\subset [2\ell]$, $y_S\in (\bigwedge(V_{2 \ell}^*))^*$ be defined by $y_S(x_{S'})=\delta_{S,S'}$ for $S'\subset [2\ell]$.
Then $R$ can be identified with the ring of polynomials in the variables $y_S$ with $S\subseteq [2\ell]$. 

The \emph{symplectic group} $\text{Sp}_{2 \ell}$ is the group of $2\ell\times 2\ell$ matrices that preserve the skew-symmetric form; i.e., for $g\in \C^{2\ell\times 2\ell}$, $g\in\text{Sp}_{2 \ell}$ if and only if $\la gx,gy \ra=\la x,y\ra$ for all $x,y\in V_{2 \ell}$.
Note that the symplectic group has a natural action on $V_{2\ell}^*$: for $g\in \text{Sp}_{2 \ell}$, $x\in V_{2\ell}^*$ and $v\in V_{2 \ell}$, $(gx)(v):=x(g^{-1}v)$. This clearly extends to an action on $\bigwedge(V_{2\ell}^*)$.
The symplectic group also has a natural action on $R$: for $g\in \text{Sp}_{2 \ell}$, $p\in R$ and $x\in \bigwedge(V^*_{2 \ell})$, $(gp)(x):=p(g^{-1}x)$.

For a set $X$, we denote by $\C X$ the vector space of finite formal $\C$-linear combinations of elements of $X$.
We will define a linear map $p:\C\D\to R$ that will turn out to have as image the space of $\text{Sp}_{2\ell}$-invariant polynomials.
To do so, we need a definition.
For an injective map $\phi:[k]\to [2\ell]$, we define the \emph{sign} of $\phi$ to be the sign of the permutation $\pi\in S_k$ such that, for $\psi=\phi\circ \pi$, we have $\psi(1)<\psi(2)<\cdots<\psi(k)$. This is denoted by $\text{sgn}(\phi)$.
For $\phi:[k]\to [2\ell]$, we define elements of $R$ by
\begin{equation}
z_{\phi}:=\left \{\begin{array}{l}\text{sgn}(\phi)y_{\phi([k])} \text{ if $\phi$ is injective,}
\\ 0\text{ otherwise.} \end{array}\right.
\label{eq:z variable} 
\end{equation}
We now define the linear map $p:\C\D\to R$ by 
\begin{equation}
G=(V,A,\kappa)\mapsto \sum_{\phi:A\to [2\ell]}(-1)^{o(\phi)}\prod_{v\in V}z_{\phi_v}
\end{equation}
for $G\in\D$.
Note that, if we evaluate $p(G)$ at $h\in \bigwedge(V^*_{2 \ell})$, we just get the skew-partition function of $h$ as defined in \eqref{eq:part skew}, since $z_{\phi_v}(h)=h(\phi_v)$. 
So $p(G)(h)=s_h(G)$.

To describe the kernel of the map $p$ we need some definitions.
For $G=(V,A,\kappa)\in \D$ and $U\subseteq A$ of size $\ell+1$, we identify $U$ with $[\ell+1]$. For $\pi \in S_{2\ell+2}$, let $A_\pi$ be the set of arcs obtained from $A$ by first removing the arcs in $U$ from $A$ and secondly by labeling for $i\in [\ell+1]$ the starting vertex of arc $i$ with $i$ and its terminating vertex with $\ell+1+i$ (some vertices may obtain multiple labels this way). 
Finally, we add $(\pi(i),\pi(\ell+1+i))$ to $A\setminus U$ for each arc $i\in[\ell+1]$.
Let $G_{U,\pi}=(V,A_\pi,\kappa)$. Let $\mathcal{J}_\ell \subset \C\D$ be the subspace spanned by 
\begin{equation}
\{\sum_{\pi \in S_{2\ell+2}}\text{sgn}(\pi)G_{U,\pi}\mid G=(V,A,\kappa)\in \D,U\subseteq A, |U|=\ell+1 \}.\label{eq:Im}
\end{equation}
Note that while $G_{U,\pi}$ may depend on the choice of identification of $U$ with $[\ell+1]$, the sum $\sum_{\pi \in S_{2\ell+2}}\text{sgn}(\pi)G_{U,\pi}$ does not. So the space $\mathcal{J}_\ell$ is well defined.

Call $G'$ a \emph{negative flip} of $G$ if $G'$ is obtained from $G$ by flipping the direction of an odd number of its arcs. 
The group $\prod_{v\in V} S_{\deg(v)}$ acts on locally ordered directed graphs with the degree sequence $(\deg(v))_{v\in V}$ by setting $\pi \cdot (V,A,\kappa)=(V,A,\kappa\circ \pi^{-1})$ for $\pi\in \prod_{v\in V} S_{\deg(v)}$ and $G=(V,A,\kappa)\in \D$.
Call $G'$ a \emph{negative permutation} of $G$ if $G'=\pi\cdot G$ for an odd permutation $\pi$.
Let $\mathcal I\subset \C\D$ be the subspace spanned by 
\[
\{G+G',G+G''\mid G\in \D, G' \text{ a negative flip of }G,G'' \text{ a negative permutation of } G\}
\]
Let finally
\begin{equation}
\mathcal{I}_\ell=\mathcal{J}_\ell+\mathcal I.\label{eq:ker p}
\end{equation}

\begin{proposition}\label{prop:im and ker p}
The image of $p$ is equal to $R^{\text{Sp}_{2 \ell}}$, the space of $\text{Sp}_{2 \ell}$-invariant polynomials in $R$, and the kernel of $p$ is equal to $\mathcal{I}_\ell$.
\end{proposition}
We postpone the proof of this proposition to the next section. First we will utilize it to characterize which invariants of locally ordered directed graph are skew-partition functions of skew-symmetric tensors.

\begin{theorem}\label{thm:char od}
An invariant of locally ordered directed graphs $f:\D\to \C$ is the partition function of a skew-symmetric tensor $h\in \bigwedge(V^*_{2 \ell})$ if and only if $f$ is multiplicative and the linear extension of $f$ to $\C\D$ is such that $f(\mathcal{I}_\ell)=0$.
\end{theorem}
\begin{proof}
Skew-partition functions are clearly multiplicative. 
By Proposition \ref{prop:im and ker p}, we have that for any $h\in \bigwedge(V^*_{2 \ell})$ that $s_h(\mathcal{I}_\ell)=0$, proving the `only if' part.

The proof of the `if' direction is based on a beautiful, and by now, well-known idea of Szegedy \cite{Sz7}; see also \cite{DGLRS12}.
We will sketch the proof.
The idea is to use the Nullstellensatz to find a solution $h\in \bigwedge(V^*_{2 \ell})$ to the set of equations $f(G)=p(G)(h)$, with $G\in \D$.
Since $f$ is multiplicative and maps $\mathcal{I}_\ell$, the kernel of $p$, to zero, there is a unique algebra homomorphism $\hat f:R^{\text{Sp}_{2 \ell}}\to \C$ such that $f=\hat f\circ p$.
If there is no solution $h\in \bigwedge(V^*_{2 \ell})$ to the set of equations $f(G)=p(G)(h)$ with $G\in \D$, then by Hilbert's Nullstellensatz, $1$ is contained in the ideal generated by $f(G)-p(G)$.
In other words, there exist locally ordered directed graphs $G_1,\ldots,G_n$ and $r_1,\ldots,r_n\in R$ such that 
\begin{equation}\label{eq:1 ideal}
1=\sum_{i=1}^n r_i(f(G_i)-p(G_i)).
\end{equation}
As the image of $p$ is equal to $R^{\text{Sp}_{2 \ell}}$, applying the Reynolds operator of the symplectic group (i.e., the projection $R\mapsto R^{\text{Sp}_{2 \ell}}$ onto the space of the $\text{Sp}_{2 \ell}$-invariants) to both sides of \eqref{eq:1 ideal}, we may assume that each $r_i$ is equal to $p(\eta_i)$ for some linear combination of locally ordered directed graphs $\eta_i$.
Now applying $\hat f$ to both sides of \eqref{eq:1 ideal} we obtain
\[
1=\sum_{i=1}^n \hat f(p(\eta_i))(f(G_i)-\hat f(p(G_i)))=\sum_{i=1}^n \hat f(p(\eta_i))(f(G_i)-f(G_i))=0,
\]
a contradiction. This finishes the proof.
\end{proof}
\begin{rem}
We have not used the circle $\bigcirc$ in our characterization, but if we want to allow the circle $\bigcirc$ as a locally ordered directed graph, then, for a $2\ell$-color edge-coloring model $h$, we should define $s_h(\bigcirc)=-2\ell$ to be compatible with our characterization.
Following up on this, in \cite{S15a} Schrijver characterizes weight systems $f$ coming from a representation of a Lie algebra (i.e., functions on certain $4$-regular locally ordered directed graphs with two incoming and two outgoing arcs at each vertex that satisfy a certain relation) in terms of rank growth of associated connection matrices. 
We note here that the condition that $f(\bigcirc)\geq 0$ should be added to statement (iv) of his theorem, since skew-partition functions can give examples of weight systems that satisfy the rank growth, but have $f(\bigcirc)< 0$.
\end{rem}

\subsection{Proof of Proposition \ref{prop:im and ker p}}
To prove Proposition \ref{prop:im and ker p} we will make use of the first fundamental theorem of invariant theory for the symplectic group and a result of Hanlon and Wales \cite{HW89}.

For a graph $G=(V,A,\kappa)\in \D$, let $n:=|V|$, identify $V$ with $[n]$ and let $D=(D_1,\ldots,D_n)$ be its \emph{degree sequence}; i.e., $D_i=\deg(i)$.
For a degree sequence $D$, we let $\D_D$ be the set of graphs with degree sequence $D$, and we let $R_D$ be the space of polynomials where each monomial is equal to $\prod_{i=1}^n y_{S_i}$, for some subsets $S_i\subseteq [2\ell]$ with $|S_i|=D_i$.
Let $2m:=\sum_{i=1}^nD_i$.
Write $p_D$ for the restriction of $p$ to $\C\D_D$.
To prove Proposition \ref{prop:im and ker p} it suffices to show that 
\begin{equation}
\text{im } p_D=(R_D)^{\text{Sp}_{2 \ell}} \text{ and }\ker p_D=\C\D_D\cap \mathcal{I}_\ell.\label{eq:restrict D}
\end{equation}

To show \eqref{eq:restrict D}, let $\oM_{2m}$ denote the collection of directed perfect matchings on $[2m]$.
We will next define maps $\tau, \sigma_D$ and $\mu_D$ so as to make the following diagram commute:
\begin{equation}\label{eq:diagram}
\begin{tikzpicture}[scale=2.5]
\node (A) at (0,.75) {$\C\D_D$};
\node (B) at (1,.75) {$R_D$};
\node (C) at (0,0) {$\C\oM_{2m}$};
\node (D) at (1,0) {$V_{2 \ell}^{\otimes 2m}$.};

\draw [->]
(C) edge node[left]{$\mu_D$} (A)
(C) edge node[above]{$\tau$} (D)
(D) edge node[right]{$\sigma_D \quad $} (B)
(A) edge node[above]{$p_D$} (B);
\end{tikzpicture}
\end{equation}
Let $\tau:\C{\oM}_m\to V_{2 \ell}^{\otimes 2m}$ be the unique linear map defined by
\begin{equation}\label{eq:matching to tensor}
M\mapsto \sum_{\substack{\phi:[2m]\to [2\ell]\\ |\phi(e)|=1 \text{ for all }e\in E(M)}} \bigotimes_{j=1}^{2m} a^M_{\phi,j}
\end{equation} 
for $M\in {\oM}_m$, where $a^M_{\phi,j}$ is equal to $e_{\phi(j)}$ if $j$ is the tail of an arc of $M$, and equal to $f_{\phi(j)}$ if $j$ is the head of an arc of $M$.
Note that since the tensor $\sum_{i=1}^{2\ell} e_i\otimes f_i$ is skew-symmetric, it follows that\begin{equation}
\text{if $M'$ is obtained from $M$ by flipping $c$ of its arcs, then $\tau(M')=(-1)^c\tau(M)$.}\label{eq:arc flip}
\end{equation}

To define the map $\mu_D$, let $\mathcal P=\{P_1,\ldots, P_n\}$ be the partition of the set $[2m]$ where $P_1:=\{1,\ldots,D_1\}$, $P_2:=\{D_1+1,\ldots,D_1+D_2\}$, etc.
Then $\mu_D$ is the unique linear map defined by sending a directed matching $M$ to the graph $G$ obtained from $M$ by identifying the vertices in each class of $\mathcal P$.
For each vertex $v$ of $G$, the bijection $\kappa_v:[\deg(v)]\to \delta(v)$ is obtained from the (natural) ordering of $[2m]$.

The map $\sigma_D$ is the unique linear map defined by sending $e_{\phi(1)}\otimes \cdots \otimes e_{\phi(2m)}$, for $\phi:[2m]\to [2\ell]$, to the monomial $\prod_{i=1}^n z_{\phi_i}$, where $\phi_i:[D_i]\to [2\ell]$ is defined as $\phi_i(j):=\phi(j+D_1+\dots+D_{i-1})$, and where $z_{\phi_i}$ is defined in \eqref{eq:z variable}.

Writing out the definition of $\sigma_D(\tau(M))$ for $M\in \oM_{2m}$, it is clear that it is equal to $p_D(\mu_D(M))$. 
So the diagram \eqref{eq:diagram} commutes. 
In particular, since $\mu_D$ is surjective, it implies by \eqref{eq:arc flip} that if $G'$ is obtained from $G$ by flipping the direction of an arc, then $p(G)=-p(G')$.

The symmetric group $S_{2m}$ has a natural action $V_{2 \ell}^{\otimes 2m}$. It also acts naturally on $\oM_{2m}$; for $\pi\in S_{2m}$ and $M\in \oM_{2m}$ we define $\pi M\in \oM_{2m}$ to be the directed matching with arcs $(\pi(i),\pi(j))$ for $(i,j)\in E(M)$.
Let us denote by $S_n\subseteq S_{2m}$ the subgroup that permutes the sets $P_1,\ldots, P_n$, but maintains the relative order of the $P_i$.
Then $S_n$, as a subgroup of $S_{2m}$, has a natural action on $\C\oM_{2m}$ and $V^{\otimes 2m}_{2\ell}$.

Let $S_D:=\prod_{i=1}^n S_{P_i}\subseteq S_{2m}$. 
We let $S_D$ act on $\C\oM_{2m}$ and $V^{\otimes 2m}_{2\ell}$ with the sign representation; i.e., for $\pi\in S_D$, $M\in \C\oM_{2m}$, and $v\in V^{\otimes 2m}_{2\ell}$, we set $\pi\cdot M:=\text{sgn}(\pi)\pi M$ and $\pi\cdot v:=\text{sgn}(\pi)\pi v$.
Then $\tau$ is $S_n$ and $S_D$-equivariant.

By the first fundamental theorem for the symplectic group, cf. \cite[Section 5.3.2]{GW09}, we have that
\begin{equation}
 \tau(\C{\oM}_m)=(V_{2 \ell}^{\otimes 2m})^{\text{Sp}_{2 \ell}}.	\label{eq:FFT}
\end{equation}
Let now $q\in R_D^{\text{Sp}_{2 \ell}}$. 
Then, as $\sigma_D$ is surjective, there exists $v\in V_{2 \ell}^{\otimes 2m}$ such that $\sigma_D(v)=q$.
Since $\sigma_D$ is $\text{Sp}_{2 \ell}$-equivariant (it is even $\text{GL}_{2\ell}$-equivariant), we may assume that $v$ is $\text{Sp}_{2 \ell}$-invariant.
So, by \eqref{eq:FFT}, there exists $M \in \C{\oM}_{2m}$ such that $\tau_D(M)=v$. 
Then, by the commutativity of \eqref{eq:diagram}, it follows that $p_D(\mu_D(M))=q$, showing $`\supseteq'$ of the first part of \eqref{eq:restrict D}. 
As, by \eqref{eq:FFT}, the commutativity of \eqref{eq:diagram} and the fact $\sigma_D$ is $\text{Sp}_{2 \ell}$-equivariant, the inclusion $\subseteq$ is also clear, we have the first part of \eqref{eq:restrict D}.

Suppose now that $p_D(\gamma)=0$ for some $\gamma\in \C\D_D$.
Then, since $\mu_D$ is surjective, we know there exists $M\in \C\oM_{2m}$ such that $\mu_D(M)=\gamma$.
Since $\gamma$ is invariant under the action of $S_n$ (as the vertex sets of the graphs in $\D_D$ are unlabeled), we may assume that $M$ is $S_n$-invariant.
Let $v:=\tau(M)$. 
Then, as $\tau$ is $S_n$-equivariant, $v$ is $S_n$-invariant. 
By the commutativity of \eqref{eq:diagram}, we have that $\sigma_D(v)=0$.
This implies that $v$ is contained in the space spanned by 
\[
\{u-\pi\cdot u\mid u\in V_{2 \ell}^{\otimes 2m}, \pi\in S_D\}.
\]
As $\C\oM_{2m}/\ker \tau \cong \tau(\C \oM_{2m})$, and $\tau$ is $S_D$-equivariant, we find that $M=M_1+M_2$ with $M_2\in \ker \tau$ and $M_1$ contained in the space spanned by 
\[
\{u-\pi\cdot u\mid u\in \C\oM_{2m}, \pi\in S_D\}.
\]
Clearly, $\mu_D(M_1)$ is contained in $\mathcal{I}_\ell$. We will now describe the kernel of $\tau$ to show that $\mu_D(M_2)$ is contained in $\mathcal{I}_\ell$ as well.

Consider two directed matchings $M,N$. Then their union is a collection of cycles (if two edges of the matchings coincide, we see this as a cycle of length two). Fix for each cycle an orientation and call an arc $a\in E(M)\cup E(N)$ \emph{odd} if it is traversed in the opposite direction. 
The number of odd arcs in $E(M)\cup E(N)$ is denoted by $o(M\cup N)$. 
Note that the parity of the number of odd arcs is independent of the choice of orientation of the cycles, as all cycles have even length.

Let us fix $M_0$ to be the directed matching with arcs $(1,2),(3,4),\ldots,(2m-1,2m)$
and define 
\begin{equation}
v_0:=\sum_{\rho\in S_{2\ell+2}}\text{sgn}(\rho)\rho M_0,	\label{eq:ker generator}
\end{equation}
where we consider $S_{2\ell+2}$ as a subgroup of $S_{2m}$ acting on $[2\ell+2]\subseteq [2m]$.
We call a matching $N$ a \emph{negative flip} of a matching $M$ if it is obtained from $M$ by flipping the directions of an odd number of arcs of $M$. 
\begin{lemma}\label{lem:ker tau}
The kernel of the map $\tau$ is spanned by 
\[
\{\pi v_0, M+M' \mid \pi \in S_{2m}, M\in \oM_{2m},\- M'\text{ a negative flip of }M\}.
\]
\end{lemma}
\begin{rem}
One can view this lemma as the tensor version of `the second fundamental theorem' for the symplectic group. 
We note that for undirected matchings the second fundamental theorem is well known. See for example \cite{LZ14, RW15} and the references in there. As far as we know, the current version for directed matchings has never appeared before.
\end{rem}
Before proving this lemma, let us first observe that it implies that the kernel of $\tau$ is mapped onto $\mathcal{I}_\ell\cap \C\D_D$. 
This implies that $\gamma\in \mathcal{I}_\ell$, showing that $\ker p_D\subseteq I$.
It is conversely straightforward using \eqref{eq:diagram} to see that $\mathcal{I}_\ell\cap \C\D_D\subseteq \ker p_D$.

So we are done by giving a proof of Lemma \ref{lem:ker tau}.
\begin{proof}[Proof of Lemma \ref{lem:ker tau}]
To prove the lemma we will use a result of Hanlon and Wales \cite{HW89}, which deals with matrices indexed by undirected matchings. We will start by relating the kernel of the map $\tau$ to the kernel of an associated matrix.

We define two matrices, $A$ and $B$, indexed by $\oM_{2m}$ as follows:
\[
A_{M,N}:=(-2\ell)^{c(M\cup N)}\quad \text{and}\quad B_{M,N}:=(2\ell)^{c(M\cup N)}(-1)^{o(M\cup N)+m},
\]
for $M,N\in \oM_{2m}$,  where $c(M\cup N)$ denotes the number of cycles in $M\cup N$.
Note that both $A$ and $B$ are $S_{2m}$-invariant.
We will now show:
\begin{equation}
\text{ the kernel of $\tau$ is equal to the kernel of $B$. }\label{eq:kernel}
\end{equation}
To prove \eqref{eq:kernel}, we first note that the bilinear form $\la\cdot,\cdot\ra$ restricted to the image of $\tau$ is nondegenerate.
Indeed, let $v\in (V_{2 \ell}^{\otimes 2m})^{\text{Sp}_{2 \ell}}$ be nonzero. 
Then we know that there exists $u\in V_{2 \ell}^{\otimes 2m}$ such that $\la v,u\ra\neq 0$.
Since for each $g\in \text{Sp}_{2 \ell}$ we have $\la v,gu\ra =\la g^{-1}v,u\ra =\la v,u\ra$, we see that, by integrating $\la v,gu\ra $ over a maximal compact subgroup of $\text{Sp}_{2 \ell}$ with respect to the Haar measure (that is, by applying the Reynolds operator), we may assume that $u$ is $\text{Sp}_{2 \ell}$-invariant. In other words, there exists $u\in (V_{2 \ell}^{\otimes 2m})^{\text{Sp}_{2 \ell}}$ such that $\la v,u \ra \neq0$, as desired.

Now fix $M,N\in \oM_{2m}$. 
We will show that $\la\tau(M),\tau(N)\ra=B_{M,N}$.
Fix for each cycle in $M\cup N$ an orientation. Flip all odd arcs in $M\cup N$ to obtain directed matchings $M'$ and $N'$ respectively.
By \eqref{eq:arc flip}, we have $\la\tau(M),\tau(N)\ra=(-1)^{o(M\cup N)}\la\tau(M'),\tau(N')\ra$.
Now
\begin{equation}
\la \tau(M'),\tau(N')\ra=\sum_{\substack{\phi,\psi:[2m]\to[2\ell]\\|\phi(e)|=1\text{ for all }e\in E(M')\\|\psi(e')|=1\text{ for all }e'\in E(N')}}\prod_{j=1}^{2m}\la a^{M'}_{\phi,j},a^{N'}_{\psi,j}\ra. \label{eq:symplectic matching}
\end{equation}
Since $M'\cup N'$ does not contain any odd arcs we see that 
\begin{equation}\label{eq:head tail}
\la a^{M'}_{\phi,j},a^{N'}_{\psi,j}\ra=\left \{\begin{array}{ll} \langle e_{\phi(j)},f_{\psi(j)} \rangle \text{ if } j \text{ is the tail of an arc of } M'
\\
\langle f_{\phi(j)},e_{\psi(j)} \rangle \text{ if } j \text{ is the head of an arc of } M'.
\end{array}\right.
\end{equation}
For each cycle of $C$ of $M'\cup N'$ we see that, to get a nonzero contribution, we need for each $j\in V(C)$ that $\psi(j)=\phi(j)$.
So each cycle $C$ gives a contribution of $(-1)^{|V(C)|/2}2\ell$. 
Hence $\la \tau(M'),\tau(N')\ra=(-1)^m (2\ell)^{c(M'\cup N')}$.
We conclude that $\la \tau(M),\tau(N)\ra=B_{M,N}$, as desired.
Now since the form $\la\cdot,\cdot\ra$ is nondegenerate on the image of $\tau$, we find that the kernel of $B$ is indeed equal to the kernel of $\tau$, proving \eqref{eq:kernel}.

It follows from Hanlon and Wales \cite[Theorem 3.1]{HW89} that the kernel of $A$ is spanned by the set
\begin{equation}
\{\sum_{\rho\in S_{2\ell+2}}\pi\rho M_0,\- M-M'\mid \pi\in S_{2m},  M\in \oM_{2m},\-M' \text{ a flip of }M\}.\label{eq:ker A}
\end{equation}
We will now show how to relate the kernel of $A$ to the kernel of $B$.

Define, for a matching $M$, $\text{sgn}(M)$ to be the sign of any permutation $\pi$ such that $\pi M=M_0$.
This is well defined since the stabilizer of $M_0$ in $S_{2m}$ consists of permutations of even sign.
First we show that for any $M,N\in \oM_{2m}$ we have:
\begin{equation}
(-1)^{o(M\cup N)}=\text{sgn}(M)\text{sgn}(N)(-1)^{c(M\cup N)}.\label{eq:signed matchings}
\end{equation}
To see this we may assume that the underlying  graph of $M$ is equal to the underlying graph of $M_0$ (by applying the same permutation to both $M$ and $N$). 
Choose an orientation of the cycles in $M\cup N$.
Multiplying both sides of \eqref{eq:signed matchings} by $(-1)^{o(M\cup N)}$, we may assume that $o(M\cup N)=0$.
Then, for each cycle $C$ in $M\cup N$, we need to apply a cycle $\pi_C\in S_{2m}$ to $N$ in order to map $N$ to $M$.
Since these cycles are of even length and pairwise disjoint, it follows that the product of their signs is equal to $(-1)^{c(M\cup N)}$. Hence $\text{sgn}(N)=(-1)^{c(M\cup N)}\text{sgn}(M)$, proving \eqref{eq:signed matchings}.

This implies for any $\mu:\oM_{2m}\to \C$:
\begin{equation}
\sum_{M\in \oM_{2m}}\mu_M M\in \ker A \iff \sum_{M\in \oM_{2m}}\text{sgn}(M)\mu_M M\in \ker B.\label{eq:ker A ker B}
\end{equation}
Indeed, define $\hat \mu:M\in \oM_{2m}\to \C$ by $\hat \mu(M)=\text{sgn}(M)\mu_M$ for $M\in\oM_{2m}$.
Then for any fixed matching $N$ we have:
\begin{eqnarray*}
(-1)^m(B\hat \mu )_N=\sum_{M\in \oM_{2m}}\text{sgn}(M)(-1)^m\mu_MB_{N,M}=\text{sgn}(N)\sum_{M\in \oM_{2m}}\mu_M A_{N,M}.
\end{eqnarray*}
Now combining \eqref{eq:ker A} with \eqref{eq:ker A ker B} we obtain the Lemma.
\end{proof}

\section{Proof of Theorem \ref{thm:main}}\label{sec:proof}
\subsection{The equivalence (i) $\Leftrightarrow$ (ii)}
We will derive the equivalence of (i) and (ii) from Theorem \ref{thm:char od}.
Let $\E\subset \G$ be the collection of Eulerian graphs, where we do not allow circles and let $\overrightarrow{\E}\subset \D$ be the collection of locally ordered directed graphs where the indegree is equal to the outdegree for each vertex.
Let $H=(V,E)\in \E$ and $\omega$ be an Eulerian orientation of $H$ and let $\kappa$ be a compatible local ordering. 
When considering $(H,\omega,\kappa)$ as a locally ordered directed graph, we will refer to it as $\iota_{\omega,\kappa}(H)$.
Let now $\omega'$ be another Eulerian orientation of $H$, with compatible local ordering $\kappa'$. 
Then, by the proof of Proposition \ref{prop:inv keuze}, we have that $(-1)^{c(H,\kappa)}\iota_{\omega,\kappa}(H)=(-1)^{c(H,\kappa')}\iota_{\omega',\kappa'}(H)\mod \I$.
This implies that we have a unique map 
\[
\iota: \E\to\C\D/\mathcal I \quad \text{ given by }\quad H\mapsto (-1)^{c(H,\kappa)}\iota_{\omega,\kappa}(H)
\]
for any choice of Eulerian orientation $\omega$ and compatible local ordering $\kappa$. 
It is clear that, extending $\iota$ linearly to $\C\E$, we obtain a linear bijection $\iota:\C\E\to \C\overrightarrow{\E}/(\mathcal I\cap\C\overrightarrow{\E})$.
So, the composition $p\circ \iota:\C\E\to R$ is a well-defined map.
Let $I_\ell\subset \C\E$ be the linear space spanned by
\begin{equation}
\{\sum_{\pi\in S_{2\ell+2}}H_{U,\pi}\mid H=(V,E)\in \E, U\subseteq E, |U|=\ell+1\}.
\end{equation}
We will now show that
\begin{equation}\label{eq:ker p na iota}
\ker (p\circ \iota)= I_\ell.
\end{equation}
Since any locally ordered directed graph $G\in\overrightarrow{\E}$ is congruent to $\iota_{\omega,\kappa}(H)$ or $-\iota_{\omega,\kappa}(H)$ modulo the ideal $\mathcal I$ for some $(H,\omega,\kappa)\in \G^\#$, and since $\iota:\C\E\to \C\overrightarrow{\E}/(\mathcal I\cap \C\overrightarrow{\E})$ is a bijection, in order to prove \eqref{eq:ker p na iota}, it suffices to show that, for any $G=(V,A,\kappa)\in \overrightarrow{\E}$ and $U\subseteq A$ of size $\ell+1$ and underlying graph $H=(V,E)\in \E$, we have
\begin{equation}
\sum_{\pi \in S_{2\ell+2}}\iota(H_{U,\pi})=\pm \sum_{\pi\in S_{2\ell+2}}\text{sgn}(\pi)G_{U,\pi}\mod \mathcal I.\label{eq:iota matching som}
\end{equation}
Since $\iota(H)=\pm G\mod \mathcal I$, we may assume that there exists an Eulerian orientation $\omega$ for $H$ and a compatible local ordering $\kappa$, such that $\iota_{\omega,\kappa}(H)=G$.
To compute $H_{U,\pi}$ for $\pi\in S_{2\ell+2}$ it is convenient to give the endpoints of the edges in $U$ the same labels as in $G$ (the sum $\sum_{\pi\in S_{2\ell+2}}H_{U,\pi}$ does not depend on the chosen direction of the edges of $U$). 
We will show 
\begin{equation}
\iota(H_{U,\pi})=\text{sgn}(\pi)(V,A_\pi,\kappa)\mod \mathcal I.	\label{eq:image iota}
\end{equation}
This clearly implies \eqref{eq:iota matching som}.
It suffices to show \eqref{eq:image iota} for the case that $\pi$ is equal to the transposition $(i,j)$ for $i,j\in [2\ell+2]$ (as transpositions generate the symmetric group).
The Eulerian orientation $\omega$ and compatible local ordering $\kappa$ in $H$ naturally induce a Eulerian orientation $\omega'$ and compatible local ordering $\kappa'$ in $H_{U,\pi}$.
Then it is easy to see that $(-1)^{c(H,\kappa)}=(-1)^{c(H_{U,\pi},\kappa')+1}$, as either two circuits in $H$ are transformed into a single circuit in $H_{U,\pi}$ or conversely a circuit in $H$ is transformed into two circuits in $H_{U,\pi}$; see Figure \ref{fig:tikz2}.
So this proves \eqref{eq:image iota} and we conclude that we have \eqref{eq:ker p na iota}.

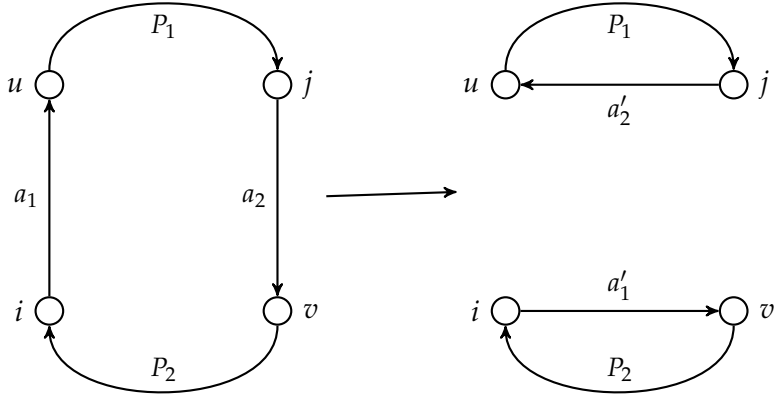
\begin{figure}[h!]
\centering
\begin{tikzpicture}[->,>=stealth',auto,node distance=3cm,
  thick,main node/.style={circle,draw,font=\sffamily\Large\bfseries}]

\node[main node] (1) [label = left : $i$]{};
\node[main node] (2) [right of=1] [label=right : $v$] {};
\node[main node] (3) [above of = 1] [label = left : $u$]{};
\node[main node] (4) [above of=2] [label=right : $j$] {};
\node[] (9) [above of =2] [yshift = -1.5cm]{};

  \path[every node/.style={font=\sffamily\small}]
    (1) edge node [left] {$a_1$} (3)
    (4) edge node[left] {$a_2$} (2)
    (2) edge[bend left=90] node [above] {$P_2$} (1)
    (3) edge[bend left=90] node [below] {$P_1$} (4);

\node[main node] (5) [right of= 2] [label = left : $i$]{};
\node[main node] (6) [right of=5] [label=right : $v$] {};
\node[main node] (7) [above of =5] [label = left : $u$]{};
\node[main node] (8) [above of=6] [label=right : $j$] {};
\node[] (10) [above of =5] [yshift = -1.5cm]{};

\draw[shorten >=0.5cm,shorten <=0.5cm,->] ([yshift=2cm] 9) -- ([yshift=2cm] 10);

  \path[every node/.style={font=\sffamily\small}]
    (5) edge node [above] {$a_1'$} (6)
    (8) edge node[below] {$a_2'$} (7)
    (6) edge[bend left=90] node [above] {$P_2$} (5)
    (7) edge[bend left=90] node [below] {$P_1$} (8);

\end{tikzpicture}
\caption{The transposition $(i,j)$} \label{fig:tikz2}
\end{figure}

Now the equivalence between (i) and (ii) follows directly from Theorem \ref{thm:char od}.
Indeed, let $f:\G\to\C$ be such that $f(G)=0$ if $G$ is not Eulerian, and extend $f$ linearly to $\C\G$. 
Then we obtain a unique linear map $\hat f:\C\D\to \C$ such that $\hat f(\iota(H))=f(H)$ for all $H\in \E$, $\hat f(H)=0$ if $H\in \D$ contains a vertex of odd (total) degree and such that $\hat f(\mathcal I)=0$.
If $f=s_h$ for some $h\in \bigwedge(V^*_{2 \ell})$, then by \eqref{eq:ker p na iota} we have that $f(I_\ell)=0$, proving that (i) implies (ii).
Conversely, if $f$ satisfies (ii), then by \eqref{eq:iota matching som} we have that $\hat f(\mathcal{I}_\ell)=0$. So there exists $h\in \bigwedge(V^*_{2 \ell})$ such that $\hat f$ is the skew-partition function of $h$. 
By definition we have that $f(G)=s_h(G)$ for $G\in \G$, finishing the proof.

\subsection{The implication (iii) $\Rightarrow$ (ii)}
Notice that, as the rank of $M_{f,0} =1$ and $f(\emptyset) = 1$, $f$ is multiplicative. 
Using the theory developed by Hanlon and Wales \cite{HW89}, we first show that $f(\bigcirc)$ has to be even.  

Let $N_{f,2k}$ be the submatrix of $M_{f,2k}$ indexed by the matchings on $2k$ vertices, i.e., the entries of $N_{f,2k}$ are of the form $f(\bigcirc)^n$ for some $n \in \N$. So $N_{f,2k}$ depends only on $f(\bigcirc)$. 
Schrijver \cite[Proposition 3]{S15} shows that if $\text{rk}(N_{f,2k}) \leq f(\bigcirc)^{2k}$ for all $k\in \N$, then $f(\bigcirc)\in \Z$.
Notice that if $f(\bigcirc) \in \N$, then $\text{rk}(N_{f,2k}) \leq f(\bigcirc)^{2k}$. 
In Section $3$ of \cite{HW89} it is explained that the eigenspaces of $N_{f,2k}$ are in one-to-one correspondence with partitions of $2k$ and that their dimension can be computed with the hook-length formula. If $f(\bigcirc) = n$, then we denote $N_{f,2k}$ by $N_{n,2k}$. 

\begin{lemma}\label{lem:rank growth}
For $n \in \N$, we have that $\text{rank}(N_{-n,2k}) \leq n^{2k}$ for all $k\in \N$ if and only if $n$ is even.
\end{lemma}
\begin{proof}
If $n= -2$, then the partition $(2,...,2)$ of $2k$ is the only partition that corresponds to an eigenspace with non-zero eigenvalue of $N_{-2,2k}$ cf. \cite[Theorem 3.1]{HW89}. By the hook-length formula, cf. \cite{Sag}, we have that 
\[
\text{rank}(N_{-2,2k}) = \frac{(2k)!}{(k+1)!k!} \leq 2^{2k}. 
\]
The matrix $N_{-2n,2k}$ is the Schur-product of $N_{-2,2k}$ and $N_{n,2k}$, so $\text{rk}(N_{-2n,2k}) \leq (2n)^{2k}$.

If $n = 2m-1$ is odd, then for even $d$ the partitions $\lambda_d = (2m+d,2m,...,2m)$ of $2md+d$ give eigenspaces with non-zero eigenvalue of $N_{n,2md+d}$, cf. \cite[Theorem 3,1]{HW89}. Denote the dimension of the corresponding eigenspace with $f_d$. Following ~\cite{S15} we have, with the hook-length formula, that
\[
f_d = \frac{(2md+d)!}{(d!)^{2m+1}p(d)}, 
\]
where $p(d)$ is a polynomial in $d$. So by Stirling's formula
\[
(f_d)^{1/(2md+d)} \rightarrow 2m+1, \text{ as }d\to \infty,
\]
showing that the rank of $N_{-n,k}$ is larger than $(n+1)^{2k}$ for $k$ large enough.
This finishes the proof.
\end{proof}
Lemma \ref{lem:rank growth} implies that $f(\bigcirc) = -2\ell$ for some $\ell\in \N$. 
To prove the remaining part of the implication we use some framework developed by Schrijver \cite{S15}, which we will now describe.
Let $F_1F_2$ be the $2k$-fragment obtained from the disjoint union of $F_1$ and $F_2$, by deleting the vertex labeled $k+i$ of $F_1$ and the vertex labeled $i$ of $F_2$ and gluing the two edges incident to these vertices together for $i=1,\ldots,k$. We can extend this bilinearly to obtain an associative multiplication on $\C\F_{2k}$.
%
This makes $\C\F_{2k}$ into an associative algebra. The unit $\bold{1}_k$ in  $\C\F_{2k}$ is given by $k$ disjoint labeled edges such that the open ends are labeled $i$ and $i+k$. 

Let $\mathcal{I}_{2k}:=\{\gamma\in \C\F_{2k}\mid f(\gamma *F)=0 \text{ for all }F\in \F_{2k}\}$. Then $\mathcal{I}_{2k}$ is an ideal and $\mathcal{A}_{k}:=\C\F_{2k}/\mathcal{I}_{2k}$ is an associative algebra. Define $\tau : \mathcal{A}_k \rightarrow \C$ for $x\in \mathcal{A}_k$ by 
\[
\tau(x) = f(x * \bold{1}_k). 
\]
Using the function $\tau$, Schrijver \cite{S15} showed that the algebra $\mathcal{A}_k$ is semisimple. Moreover, Proposition 6 in \cite{S15} implies,  
\begin{equation}
\text{ if $x$ is a non-zero idempotent in $\mathcal{A}_k$, then $\left \{\begin{array}{ll}\tau(x) \text{ is a negative integer if $k$ is odd,}\\ \tau(x)\text{ is a positive integer if $k$ is even.}\end{array}\right.$}\label{eq:idp}
\end{equation}
In fact, the case that $k$ is even in \eqref{eq:idp} is exactly Proposition 6 in \cite{S15}, as $k$ is even implies that $f(\bigcirc)^k>0$. The case that $k$ is odd can be derived in exactly the same way. 

Let $k,m \in \N$, then, following \cite{S15}, for $\pi \in S_m$, let $P_{k,\pi}$ be the $2km$-fragment consisting of $km$ disjoint edges $e_{i,j}$ for $i = 1,...,m$ and $j= 1,...,k$, where $e_{i,j}$ connects the vertices labeled $j+(i-1)k$ and $km+j +(\pi(i)-1)k$. We define $q_{k,m}$ to be  
\[
q_{k,m} := \sum_{\pi \in S_m} P_{k,\pi}.
\]
Let $o(\pi)$ be the number of orbits of the permutation $\pi$. If $m> (2 \ell)^k$ and $k$ is odd, we have that
\begin{align}
\tau(q_{k,m}) &= \sum_{\pi \in S_m} ((-2 \ell)^k)^{o(\pi)} = (-1)^{m}\sum_{\pi \in S_m} (-1)^{m-o(\pi)}(2 \ell)^{ko(\pi)}	\nonumber
\\
&= (-1)^{m} \sum_{\pi \in S_m} \text{sgn}(\pi)((2 \ell)^k)^{o(\pi)} = 0, \label{eq:idp is 0}
\end{align}
since $\sum_{\pi\in S_m}\text{sgn}(\pi)x^{o(\pi)}=x(x+1)\cdots(x-m+1)$. 
Using this, we first show that $f$ satisfies
\begin{equation}
\sum_{\pi\in S_{2\ell+2}}f(G_{U,\pi})=0 \text{ for each graph } G=(V,E)\in \E \text{ and } U\subseteq E \text{ of size } \ell+1.\label{eq: (ii)}
\end{equation}
Let $m = 2\ell+2$ and consider $q_{1,m}$. Then, by \eqref{eq:idp is 0}, we have that $\tau(q_{1,m})=0$. 
Since $\frac{1}{m!} \ q_{1,m}$ is an idempotent, \eqref{eq:idp} implies that $q_{1,m}=0$ in $\mathcal{A}_{\ell+1}$. In other words, $q_{1,m}\in\mathcal{I}_{m}$. 
Consider a graph $G = (V,E)$ and let $U \subseteq E$ of size $\ell +1$. 
Let $F \in \mathcal{F}_{2k}$ be obtained from $G$ as follows: replace each edge $e=\{u,v\}$ from $U$ by two open ends such that one is connected to $u$ and the other to $v$.
Then, as $q_{1,m} \in \mathcal{I}_{m}$,  
\[
0= f(F * q_{1,m}) = \sum_{\pi \in S_m} f(G_{U,\pi}),
\]
proving \eqref{eq: (ii)}.

To show that $f(G)=0$ if $G$ is non-Eulerian, fix a vertex $v$ of $G$ that has odd degree, say $k$.
Define fragments $F_0,F_1\in \F_k$ as follows: $F_0$ consists of a single vertex with $k$ open ends incident with this vertex; $F_1$ is obtained from $G$ by removing $v$ from $G$ and replacing each edge incident with $v$ with an open end. Then $F_0*F_1=G$.
Now take $m$ such that $m> (2 \ell)^k$. 
Then $\frac{1}{m!}q_{k,m}$ is an idempotent and by \eqref{eq:idp is 0}, $\tau(q_{k,m}) = 0$, and so, by \eqref{eq:idp}, $q_{k,m}$ is actually $0$ in $\mathcal{A}_{km}$. 
Now, take $m$ copies of both $F_0$ and $F_1$ and create a fragment $F\in \F_{km}$ from their disjoint union as follows: the end labeled $j$ in $F_i$ gets label $ikm + j+k(n-1)$ in the $n$-th copy of $F_i$. Then, as $q_{k,m}\in \I_{mk}$, 
\[
0 = f(F * q_{k,m}) = m!(f(F_0 * F_1)^m)=m!(f(G)^m). 
\]
So $f(G) = 0$, finishing the proof.

\subsection{The implication (i) $\Rightarrow$ (iii)}
Fix $k\in \N$.
For an Eulerian $2k$-fragment (i.e., all unlabeled vertices have even degree), an \emph{Eulerian orientation} is defined to be an orientation of the edges such that, for each unlabeled vertex, the number of incoming arcs equals the number of outgoing arcs and such that there are exactly $k$ labeled vertices incident with an incoming arc (note that such an orientation always exists, which can for example be see by temporarily adding an edge between the vertices labeled $i$ and $k+i$ for $i=1,\ldots,k$).

Let $\mathcal{F}_{2k}^{\#}$ be the set of triples $(F,\omega,\kappa)$, where $F$ is an Eulerian $2k$-fragment, none of whose components are circles, with $\omega$ an Eulerian orientation and $\kappa$ a local ordering compatible with $\omega$ at all non-labeled vertices.

Notice that, similar to what we have seen earlier, the local ordering $\kappa$ decomposes $E(F)$ into circuits and $k$ directed walks that begin and end at a labeled vertex. Denote by $c(F,\kappa)$ the number of circuits in this partition and let $M=M(F,\kappa)$ be the directed perfect matching on $[2k]$, where $(i,j)\in E(M)$ if and only if there is a directed walk from $i$ to $j$ in this partition.

For $(F,\omega,\kappa)\in \F^\#_{2k})$, a map $\psi:[2k]\to [2\ell]$ and a map $\phi:E(F)\to [2\ell]$, we say that $\phi$ \emph{extends} $\psi$ if, for each $i\in [2k]$ and the unique open end $e\in E(F)$ incident with the vertex labeled $i$,  we have that $\phi(e)=\psi(i)$. By $o'(\phi)$ we denote the number of edges that are not open ends and that get mapped into $[\ell]$ by $\phi$.
Now define the tensor $t_h(F,\omega,\kappa) \in V_{2 \ell}^{\otimes 2k}$ as follows:
\begin{equation}
(-1)^{c(F,\kappa)}\sum_{\psi:[2k]\to [2 \ell]}\Big(\sum_{\substack{\phi:E(F)\to[2\ell]\\\phi \text{ extends } \psi}}(-1)^{o'(\phi)}\prod_{v\in V(F)}h(\phi_v)
\Big)\bigotimes_{(i,j)\in E(M)}e_{\phi(i)}\otimes f_{\phi(j)}.
\end{equation}

Let $(F_1,\omega_1,\kappa_1)$ and $(F_2,\omega_2,\kappa_2)$ be elements of $\mathcal{F}_{2k}^{\#}$ and let $M_i=M(F_i,\kappa_i)$ for $i=1,2$. 
Fix an arbitrary orientation of the cycles in $M_1\cup M_2$ and consider the odd arcs of $M_1\cup M_2$. For $i=1,2$ and an odd arc $(a,b)\in E(M_i)$, change the orientation of $\omega_i$ by flipping the directions of the edges in the path from $a$ to $b$ in the partition of $E(F_i)$, and change the local ordering at each vertex of the path by interchanging the edges in the path incident with the vertex.
Call the resulting orientations and local orderings $\omega_i'$ and $\kappa_i'$, respectively, and note that $\kappa_i'$ is compatible with $\omega_i'$.
Now $\omega_1'$ and $\omega_2'$ induce an Eulerian orientation $\omega$ in $F_1*F_2$. Moreover, $\kappa_1'$ and $\kappa_2'$ induce a compatible local ordering $\kappa$ in $F_1*F_2$.
By definition we have 
\[
s_h(F_1*F_2)=(-1)^{k+c(M_1\cup M_2)}\langle t_h(F_1,\omega'_1,\kappa'_1),t_h(F_2,\omega'_2,\kappa'_2)\rangle,
\]
where the $(-1)^k$ factor is to compensate for the skew-symmetry of the form, cf. \eqref{eq:head tail}, and where the $(-1)^{c(M_2\cup M_2)}$ factor comes from the circuits formed by gluing $F_1$ and $F_2$.
By the same argument, as in the proof of Lemma \ref{lem:inv cycle}, we have that $t_h(F_i,\omega_i,\kappa_i)=(-1)^{o(M_i)}t_h(F_i,\omega_i',\kappa_i')$ for $i=1,2$, where $o(M_i)$ denotes the number of odd arcs of $M_i$ in $M_1\cup M_2$.
As, by \eqref{eq:signed matchings}, $(-1)^{c(M_1\cup M_2)}=(-1)^{o(M_1\cup M_2)}\text{sgn}(M_1)\text{sgn}(M_2)$, we conclude that 
\begin{equation}\label{eq:gramm}
s_h(F_1*F_2)=\text{sgn}(M_1)\text{sgn}(M_2)(-1)^k\langle t_h(F_1,\omega_1,\kappa_1),t_h(F_2,\omega_2,\kappa_2)\rangle.
\end{equation}
Now \eqref{eq:gramm} implies that we can write $M_{s_h,2k}$ as the Gramm matrix (with respect to $\langle \cdot,\cdot,\rangle$) of vectors in $V_{2 \ell}^{\otimes 2k}$, implying that its rank is bounded by $(2\ell)^{2k}$.

\section{Concluding remarks and questions}\label{sec:remarks}
{\bf Future work.}
Having introduced and characterized skew-partition functions for graphs, the story does not end here.
For $m\in \mathbb{Z}$, consider the graph parameter $f_m$ defined by 
\begin{equation}
f_m(G):=\left\{\begin{array}{l} m^{c(G)}\text{ if $G$ is $2$-regular,}
\\0 \text{ otherwise,}\end{array}\right.
\end{equation}
where $c(G)$ is the number of components of $G$. Then, if $m\in \N$, we have that $f_m$ is the (ordinary) partition function of an edge-coloring model; if $m\in -2\N$, then $f_m$ is the skew-partition function of a skew-symmetric tensor.
This leaves open the case where $m$ is negative and odd. 
Let us look at the special case that $m=-1$.
Set $f_{-1}(\bigcirc)=-1$. 
Then it can be shown, using for example the results of Halon and Wales \cite{HW89}, that $\text{rk}(M_{f_{-1},k})\leq 3^k$ for each $k$. It can also be shown that $\text{rk}(M_{f_{-1},k})>2^k$ for $k$ large enough, cf. Lemma \ref{lem:rank growth}.
This implies that $f_{-1}$ is neither a partition function nor a skew-partition function.
For a connected graph $G$ we can however realize $f_{-1}(G)$ as the sum $f_{-2}(G)+f_{1}(G)$, a sum of a skew-partition function and a partition function, and extend this multiplicatively to disjoint unions.
The associated models live in $(\C^*)^{\otimes 2}$ and $((\C^2)^*)^{\otimes 2}$, respectively.
Taking their direct sum we could think of it as living in $((\C\oplus \C^2)^*)^{\otimes 2}$.
By equipping $\C$ with a nondegenerate bilinear form and $\C^2$ with a nondegenerate skew-symmetric form, we can interpret $\C$ as the even component and $\C^2$ as the odd component of the super vector space $\C\oplus \C^2$. It is natural to try to define (and characterize) partition functions for models living in a super vector space.
In particular we believe that evaluations of the circuit partition polynomial at odd negative integers can be realized in this way. This will be the object of future study.
\\
\quad\\
{\bf Computational complexity.}
For partition functions of edge-coloring models with $2$ colors (seen as Holant problems) there is a surprising dichotomy result \cite{CGW13,W15}, saying that it is \#P hard to evaluate these unless the model satisfies some particular conditions, in which case the partition function can be evaluated in polynomial time.
Are similar results true for the skew-partition functions introduced in the present paper?
As far as we know, even the computational complexity of evaluating the cycle partition polynomial at negative even integers $2\ell$ with $\ell<1$ is still open. 
For the directed cycle partition polynomial, this is \#P hard even for regular planar graphs of degree $4$, except at $x=0$ and $x=-2$, as follows from a reduction to the Tutte polynomial. See \cite{EM,V05} for details.

\section*{Acknowledgements}
The research leading to these results has received funding from the European Research Council
under the European Union's Seventh Framework Programme (FP7/2007-2013) / ERC grant agreement
n$\mbox{}^{\circ}$ 339109.

We thank Lex Schrijver for useful comments on a earlier version of this paper. We are moreover grateful to the anonymous referees for their careful reading of the paper and their helpful comments.


\begin{thebibliography}{30}

\bibitem{Bol} B. Bollob\'as, Evaluations of the circuit partition polynomial, {\sl Journal of Combinatorial Theory, Series B} {\bf 85.2} (2002) 261-268.

\bibitem{CLX11}J. Cai, P. Lu, and M. Xia, Computational complexity of Holant problems, {\sl SIAM Journal on Computing} {\bf 40} (2011) 1101--1132.

\bibitem{CGW13}J. Cai, H. Guo, and T. Williams, A complete dichotomy rises from the capture of vanishing signatures, {\sl Proceedings of the forty-fifth annual ACM symposium on Theory of computing}, ACM, 2013.

\bibitem{CDM12}S. Chmutov, S. Duzhin, and J. Mostovoy, {\sl Introduction to Vassiliev Knot Invariants}, Cambridge University Press, 2012.

\bibitem{DGLRS12}J. Draisma, D. Gijswijt, L. Lov\'asz, G. Regts, A. Schrijver, 
Characterizing partition functions of the vertex model, {\sl Journal of Algebra} {\bf 350} (2012) 197--206.

\bibitem{EM}J.A. Ellis-Monaghan, 
Identities for circuit partition polynomials, with applications to the Tutte polynomial, {\sl Advances in Applied Mathematics} {\bf 32} (2004) 188--197.


\bibitem{FLS}M. Freedman, L. Lov\'asz, A. Schrijver, Reflection positivity, rank connectivity, and homomorphisms of graphs, {\sl Journal of the American Mathematical Society} {\bf 20} (2007) 37--51.

\bibitem{GW09}R. Goodman, N.R. Wallach, {\sl Symmetry, Representations, and Invariants}, Graduate Texts in Mathematics, Vol. 255, Springer New York, 2009.

\bibitem{HW89}P. Hanlon, and D. Wales, On the decomposition of Brauer's centralizers algebras, {\sl Journal of Algebra} {\bf 121} (1989) 409--445.

\bibitem{HJ}P. de la Harpe, V.F.R. Jones, Graph invariants related to statistical mechanical models:
examples and problems, {\sl Journal of Combinatorial Theory, Series B} {\bf 57} (1993) 207--227.

\bibitem{LZ14}G.I. Lehrer, R.B. Zhang, The second fundamental theorem of invariant theory for the orthosymplectic supergroup, arXiv preprint, arXiv:1407.1058, 2014.

\bibitem{L12}L. Lov\'asz, {\sl Large Networks and Graph Limits}, American Mathematical Society, Providence, Rhode Island, 2012.

\bibitem{LS9}L. Lov\'asz, B. Szegedy, Contractors and connectors of graph algebras, {\sl Journal of Graph Theory} {\bf 60} (2009) 11--30.

\bibitem{MS08}I. L. Markov, and Y. Shi, Simulating quantum computation by contracting tensor networks, {\sl SIAM Journal on Computing} {\bf 38} (2008) 963--981.

\bibitem{Martin}P. Martin, {\sl Enumérations eulériennes dans les multigraphes et invariants de Tutte-Grothendieck}, Diss. Institut National Polytechnique de Grenoble-INPG; Université Joseph-Fourier-Grenoble I, 1977.

\bibitem{P71}R. Penrose, Applications of negative dimensional tensors, in: {\sl Combinatorial Mathematics and Its Applications} (D.J.A. Welsh, ed.), Academic Press, London, 1971 pp. 221--244.

\bibitem{R13}G. Regts, {\sl Graph Parameters and Invariants of the Orthogonal Group}, PhD thesis, University of Amsterdam, 2013.

\bibitem{RSS15} G. Regts, A. Schrijver, B. Sevenster, On the existence of real R-matrices for virtual link invariants, to appear in {\sl Halin Memorial Volume}, arXiv preprint, arXiv:1503.01882, 2015 

\bibitem{RW15}M. Rubey, B.W. Westbury, Combinatorics of symplectic invariant tensors, arXiv preprint, arXiv:1504.02586, 2015.

\bibitem{Sag}B.E. Sagan, {\sl The Symmetric Group: Representations, Combinatorial Algorithms, and Symmetric Functions}, Graduate Texts in Mathematics, Vol. 203, Springer Science \& Business Media, 2001.

\bibitem{S9} A. Schrijver, Graph invariants in the spin model, {\sl Journal of Combinatorial Theory, Series B} {\bf 99} (2009) 502--511.

\bibitem{S15}A. Schrijver, Characterizing partition functions of the edge-coloring model by rank growth, {\sl Journal of Combinatorial Theory, Series A} {\bf 136} (2015) 164--173.

\bibitem{S15a}A. Schrijver, Connection matrices and Lie algebra weight systems for multiloop chord diagrams, {\sl Journal of Algebraic Combinatorics} {\bf 42} (2015) 896--905.

\bibitem{S15b}A. Schrijver, On traces of tensor representations of diagrams, {\sl Linear Algebra and its Applications} {\bf 278} (2015) 28-41.


\bibitem{Sz7}B. Szegedy, Edge-coloring models and reflection positivity, {\sl Journal of the American Mathematical Society} {\bf 20} (2007) 969--988.

\bibitem{Sz10}B. Szegedy, Edge coloring models as singular vertex-coloring models, in: {\sl Fete of Combinatorics and Computer Science} (G.O.H. Katona, A. Schrijver, T.Sz\"onyi, editors), Springer, Heidelberg and J\'anos Bolyai Mathematical Society, Budapest (2010) 327--336.

\bibitem{V05}D. Vertigan, The computational complexity of Tutte invariants for planar graphs, {\sl SIAM Journal on Computing} {\bf 35} (2005) 690--712.

\bibitem{W15}T. Williams, {\sl Advances in the Computational Complexity of Holant Problems}, PhD thesis, University of Wisconsin-Madison, 2015.

\end{thebibliography}
\end{document}